\documentclass[12pt]{article}
\usepackage[colorlinks]{hyperref}
\usepackage{color}
\usepackage{graphicx}
\usepackage{graphics}
\usepackage{makeidx}
\usepackage{showidx}
\usepackage{latexsym}
\usepackage{amssymb}
\usepackage{verbatim}
\usepackage{amsmath}
\usepackage{amsthm}
\usepackage{amsfonts}
\usepackage{amssymb,amsmath}
\usepackage{latexsym,amsthm,amscd}
\usepackage[all]{xy}
\newtheorem{defn}{Definition}[section]
\newcounter{alphthm}
\numberwithin{equation}{section}

\newtheorem{rem}{Remark}[section]

\newtheorem{thm}{Theorem}
\newtheorem{theorem}{Theorem}[section]
\newtheorem{prop}{Proposition}[section]
\newtheorem{lem}{Lemma}[section]
\newcommand{\be}{\begin{equation}}
\newcommand{\ee}{\end{equation}}
\newcommand{\ben}{\begin{enumerate}}
\newcommand{\een}{\end{enumerate}}
\newcommand{\beq}{\begin{eqnarray}}
\newcommand{\eeq}{\end{eqnarray}}
\newcommand{\beqn}{\begin{eqnarray*}}
\newcommand{\eeqn}{\end{eqnarray*}}

\newcommand{\bpf}{\begin{proof}}
\newcommand{\epf}{\end{proof}}
\newcommand{\bl}{\begin{lem}}
\newcommand{\el}{\end{lem}}
\newcommand{\bp}{\begin{prop}}
\newcommand{\ep}{\end{prop}}
\newcommand{\bd}{\begin{defn}}
\newcommand{\ed}{\end{defn}}
\newcommand{\bt}{\begin{thm}}
\newcommand{\et}{\end{thm}}

\newcommand\bpr{\begin{prop}}
\newcommand\epr{\end{prop}}
\title{The Axiom of Spheres in Finsler Geometry}
\author{M. K. Sedaghat and B. Bidabad\thanks{The corresponding author, bidabad@aut.ac.ir}}
\date{}
\begin{document}
\maketitle
\noindent

\begin{abstract}
 Here, an axiom of spheres in Finsler geometry is proposed and it is proved that if a Finslerian manifold satisfies the axiom of spheres then it is of constant flag curvature.
 \end{abstract}
\vspace{.5cm}
\section*{Introduction}
In Riemannian geometry, E. Cartan defined an axiom of $r$-planes as follows. \emph{A Riemannian manifold $M$ of dimension $n\geq3$ satisfies the axiom of $r$-planes, where $r$ is a fixed integer $2\leq r< n$, if for each point $p$ of $M$ and any $r$-dimensional subspace $S$ of the tangent space $T_{p}M$ there exists an $r$-dimensional totally geodesic submanifold $V$ containing $p$ such that $T_{p}V=S$.} He proved that if $M$ satisfies the axiom of $r$-planes for some $r$, then $M$ has constant sectional curvature, cf., \cite{C}. The axiom of $r$-spheres in Riemannian geometry was proposed by D. S. Leung and K. Nomizu as follows. \emph{For each point $p$ of $M$ and any $r$-dimensional subspace $S$ of $T_{p}M$, there is an $r$-dimensional umbilical submanifold $V$ with parallel mean curvature vector field such that $p\in V$ and $T_{p}V=S$.} They proved that if a Riemannian manifold $M$ of dimension $n\geq3$ satisfies the axiom of $r$-spheres for some $r$, $2\leq r< n$, then $M$ has constant sectional curvature, cf., \cite{LN}.
In \cite{AZ}, Akbar-Zadeh extends the Cartan's axiom of 2-planes to Finsler geometry as follows. \emph{A Finslerian manifold $M$ of dimension $n\geq3$ satisfies the axiom of 2-planes if for each point $p\in M$ and every subspace $E_{2}$ of dimension two of $T_{p}M$ there exists a totally geodesic surface $S$ passing through $p$ such that $T_{p}S=E_{2}$.} He proved that every Finsler manifold satisfying the axiom of 2-planes is of constant flag curvature, cf., \cite{AZ}, page 182.

Recently, a definition of circle in Finsler spaces is introduced by one of the present authors in a joint work with Z. Shen, cf., \cite{BS}.  Based on the definition of a circle  we will show later  that a connected  submanifold of a Finsler manifold is an \emph{extrinsic sphere} if and only if its circles coincide with circles of the ambient manifold. The proof will appear elsewhere.\\

 In the present work,  we propose in a natural way, the following axiom of $r$-spheres in Finsler geometry.\\
\\
\emph{{\bf{Axiom of $r$-spheres.}} Let $(M,F)$ be a Finsler manifold of dimension $n\geq3$. For each point $x$ in $M$ and any $r$-dimensional subspace $E_{r}$ of $T_{x}M$, there exists an $r$-dimensional umbilical submanifold $S$ with parallel mean curvature vector field such that $x\in S$ and $T_{x}S=E_{r}$.}\\
\\
We shall prove the following theorem.
\begin{thm}\label{main1}
If a Finsler manifold of dimension $n\geq3$ satisfies the axiom of $r$-spheres for some $r$, $2\leq r< n$, then $M$ has constant flag curvature.
\end{thm}
\section{Notations and preliminaries on Finsler submanifolds}
Let $M$ be a real n-dimensional manifold of class $C^{\infty}$. We denote by $TM$ the tangent bundle of tangent vectors,  by  $p :TM_{0}\longrightarrow M$ the fiber bundle of non-zero tangent vectors and  by $p^*TM\longrightarrow TM_0$ the pulled-back tangent bundle.
Let $(x,U)$ be a local chart on  $ M$ and $(x^i,y^i)$ the induced local coordinates on $p^{-1}(U)$.  A \emph{Finsler structure} on M is a function $F: TM \longrightarrow [0,\infty )$, with the following properties:(i) $F$ is differentiable $C^{\infty}$ on $TM_{0}$; (ii) $F$ is positively homogeneous of degree one in $y$, that is,  $F(x,\lambda y)=\lambda F(x,y)$, for all $\lambda >0$;  (iii) The Finsler \emph{metric tensor} $g$ defined by the Hessian matrix of $F^{2}$, $(g_{ij})=(\frac{1}{2}[\frac{\partial^{2}}{\partial y^{i}\partial y^{j}}F^{2}])$, is positive definite on $TM_{0}$. A \emph{Finsler manifold} is a pair $(M,F)$ consisting of a differentiable manifold $M$ and a Finsler structure $F$ on $M$. We denote by $TTM_0$, the tangent bundle of $TM_0$ and by $\rho$, the canonical linear mapping $\rho:TTM_0\longrightarrow p^*TM,$ where, $ \rho=p_*$. There is the horizontal distribution $HTM$ such that we have the Whitney sum $TTM_0=HTM_{0}\oplus VTM_{0}.$ This decomposition permits to write a vector field $\hat{X}\in \chi(TM_0)$ into the horizontal and vertical parts in a unique manner, namely $\hat{X}=H\hat{X}+V\hat{X}$. In the sequel, we decorate the vector fields on $TM_0$ by hat, i.e. $\hat{X}$ and $\hat{Y}$ and the corresponding sections of $p^*TM$ by $X=\rho(\hat X)$ and $Y=\rho(\hat Y)$, respectively, unless otherwise specified, cf., \cite{AZ}. For all $X\in p^{*}TM$ we denote by $^h\!\hat{X}$ the horizontal lift of $X$ defined by the bundle morphism $\beta:p^{*}TM\longrightarrow HTM$ where, $\beta(\frac{\partial}{\partial x^i})=\frac{\delta}{\delta x^i}$, cf., \cite{A}. For another approach on geometry of Finslerian manifolds, one can refer to  \cite{BM}.
\subsection{Finsler geometry of submanifolds}
Let $(M,F)$ be a Finsler manifold and $S\subset M$ a $k$-dimensional submanifold defined by the immersion $i:S\longrightarrow M$. We identify any point $x\in S$ by its image $i(x)$ and any tangent vector $X\in T_{x}S$ by its image $i_{*}(X)$, where $i_{*}$ is the linear tangent mapping. Thus $T_{x}S$ becomes a sub-space of $T_{x}M$. Let $TS_{0}$ be the fiber bundle of non-zero tangent vectors on $S$. $TS_{0}$ is a sub-vector bundle of $TM_{0}$ and the restriction of $p$ to $TS_{0}$ is denoted by $q:TS_{0}\longrightarrow S$. We denote by $\bar{T}(S)=i^{*}TM$, the pull back induced vector bundle of $TM$ by $i$. The Finslerian metric $g$ on $TM_{0}$ induces a Finslerian metric on $TS_{0}$, where, we denote it again by $g$.
At a point $x=qz\in S$, where $z\in TS_{0}$, the orthogonal complement of $T_{qz}S$ in $\bar{T}_{qz}S$ is denoted by $N_{qz}S$, namely, $\bar{T}_{x}(S)=T_{x}(S)\oplus N_{qz}S,$ where $T_{x}(S)\cap N_{qz}S=0$. We have the following decomposition:
\begin{equation}\label{1}
q^{*}\bar{T}S=q^{*}TS\oplus N,
\end{equation}
where, $N$ is called the normal fiber bundle. If $TTS_{0}$ is the tangent vector bundle to $TS_{0}$, we denote by $\varrho$, the canonical linear mapping $\varrho:TTS_{0}\longrightarrow q^{*}TS$. Let $\hat{X}$ and $\hat{Y}$ be the two vector fields on $TS_{0}$. For $z\in TS_{0}$, $(\nabla_{\hat{X}}Y)_{z}$ belongs to $\bar{T}_{qz}S$. Attending to (\ref{1}) we have
\begin{equation}\label{2}
\nabla_{\hat{X}}Y=\bar{\nabla}_{\hat{X}}Y+\alpha(\hat{X},Y),\quad Y=\varrho(\hat{Y}),\quad X=\varrho(\hat{X}),
\end{equation}
where, $\nabla$ is the covariant derivative of Cartan connection and $\alpha(\hat{X},Y)$ the second fundamental form of the submanifold $S$. It belongs to $N$ and is bilinear in $\hat{X}$ and $Y$. It results from (\ref{2}) that the induced connection $\bar{\nabla}$ is a metric compatible covariant derivative with respect to the induced metric $g$ in the vector bundle $q^{*}TS\longrightarrow TS_{0}$.
\subsection{Shape operator or Weingarten formula}
Let $S$ be an immersed submanifold of $(M,F)$. For any $\hat{X}\in \chi(TS_{0})$ and $W\in \Gamma(N)$ we set
\begin{equation}\label{4}
\nabla_{\hat{X}}W=-A_{W}\hat{X}+\bar{\nabla}^{^\perp}_{\hat{X}}W,
\end{equation}
where, $A_{W}\hat{X}\in \Gamma(q^{*}TS)$ and $\bar{\nabla}^{^\perp}_{\hat{X}}W\in \Gamma(N)$ and we have partially used notations of \cite{BF}. It follows that $\bar{\nabla}^{^\perp}$ is a linear connection on the normal bundle $N$. We also consider the bilinear map
\begin{align*}
A:&\Gamma(N)\otimes\Gamma(TTS_{0})\longrightarrow \Gamma(q^{*}TS),\\
&A(W,\hat{X})=A_{W}\hat{X}.
\end{align*}
For any $W\in \Gamma(N)$, the operator $A_{W}:\Gamma(TTS_{0})\longrightarrow \Gamma(q^{*}TS)$ is called the \emph{shape operator} or the \emph{Weingarten map} with respect to $W$. Finally, (\ref{4}) is said to be the \emph{Weingarten formula} for the immersion of $S$ in $M$. We have
\begin{align*}
g(\alpha(^h\!\hat{X},Y),W)=g(A_{W}{^h\!\hat{X}},Y),
\end{align*}
where, $g$ is the Finslerian metric of $M$, $X,Y\in \Gamma(q^{*}TS)$ and $^h\!\hat{X}$ is the horizontal lift of $X$, cf., \cite{A}.
\subsection{Totally umbilical submanifolds in Finsler spaces}
The \emph{mean curvature} vector field $\eta$ of the isometric immersion $i:S\longrightarrow M$ is defined by
\begin{equation}\label{MC}
\eta=\frac{1}{n}tr_{g}\alpha(^h\!\hat{X},Y),
\end{equation}
where, $X,Y\in \Gamma(q^{*}TS)$ and $^h\!\hat{X}$ is the horizontal lift of $X$, cf., \cite{A}. We say that the mean curvature vector field $\eta$ is parallel in all directions if $\bar{\nabla}^{^\perp}_{^h\!\hat{X}}\eta=0$ for all $X\in \Gamma(q^{*}TS)$.
\begin{defn}
\cite{A} A submanifold of a Finsler manifold is said to be totally umbilical, or simply umbilical, if it is equally curved in all tangent directions.
\end{defn}
More precisely, let $i:S\longrightarrow M$ be an isometric immersion. Then $i$ is called totally umbilical if there exists a normal vector field $\xi \in N$ along $i$ such that its second fundamental form $\alpha$ with values in the normal bundle satisfies
\begin{equation}\label{TU}
\alpha(^h\!\hat{X},Y)=g(X,Y)\xi,
\end{equation}
for all $X,Y\in \Gamma(q^{*}TS)$, where $^h\!\hat{X}$ is the horizontal lift of $X$. Equivalently, $S$ is umbilical in $M$ if $A_{W}=g(W,\xi)I$ for all $W\in\Gamma(N)$ where, $I$ is the identity transformation, cf., \cite{A}.
To give an example of a totally umbilical submanifold in Finsler space, we refer to a theorem on totally umbilical submanifolds given in \cite{HYZ}. There is shown that if $(\tilde{M}^{n+1},\tilde{\alpha}+\tilde{\beta})$ is a Randers space, where $\tilde{\alpha}$ is an Euclidean metric and $\tilde{\beta}$ is a closed 1-form, then any complete and connected n-dimensional totally umbilical submanifold of $(\tilde{M}^{n+1},\tilde{\alpha}+\tilde{\beta})$ must be either a plane or an Euclidean sphere. The latter case happens only when there exist a point $\tilde{x}_0$ and a function $\lambda(\tilde{x})$ on $M^{n+1}$ such that $\tilde{\beta}=\lambda(\tilde{x})d(\parallel \tilde{x}-\tilde{x}_0\parallel_{\tilde{\alpha}}^{2})$ and the sphere is centered at $\tilde{x}_0$.\\
\hspace{-0.6cm}{\bf{Example 1.1.}} \cite{HYZ} Let $(\tilde{M}^{n+1},\tilde{F})$ be a Randers space with $\tilde{F}=\tilde{\alpha}+\tilde{\beta}$, where,
\begin{equation}
\tilde{\alpha}=\sqrt{\sum_{k=1}^{n+1}(\tilde{y}^k)^{2}},\quad \tilde{\beta}
=\sum_{k=1}^{n+1}\frac{b\tilde{x}^{k}d\tilde{x}^{k}}{\sqrt{\sum_k(\tilde{x}^k)^{2}}},\quad \forall (\tilde{x},\tilde{y})\in T\tilde{M}_0,\nonumber
\end{equation}
 $b$ is a constant and $0<\mid b\mid<1$. One can see that $d\tilde{\beta}=0$. Let
\begin{equation}
M=\{\tilde{x}\in \tilde{M}^{n+1}:\sum_{k=1}^{n+1}(\tilde{x}^{k}-\tilde{x}_{0}^{k})^2=r^2\},\nonumber
\end{equation}
and $f:(M,F)\longrightarrow (\tilde{M}^{n+1},\tilde{F})$ be an isometric immersion where $F=\alpha+\beta$ such that
\begin{equation}
\alpha=\sqrt{\sum_{k=1}^{n+1}\frac{\partial f^k}{\partial x^i}\frac{\partial f^k}{\partial x^j}y^{i}y^{j}},\quad \beta=\sum_{k=1}^{n+1}\frac{\partial f^k}{\partial x^i}\frac{b\tilde{x}^{k}y^i}{\sqrt{\sum_k(\tilde{x}^k)^{2}}},\nonumber
\end{equation}
where, $(x,y)\in TM_0$. It is obvious that $\sum_{k=1}^{n+1}(f^k(x)-\tilde{x}^{k}_{0})\frac{\partial f^k}{\partial x^i}=0$ hence if $\tilde{x}_0=0$ then $\beta=0$. On the other hand from theorem mentioned above we see that $(M,F)$ is a totally umbilical submanifold of $(\tilde{M}^{n+1},\tilde{F})$ if $\tilde{x}_0=0$. Therefore if $\tilde{x}_0=0$, the Euclidean sphere $(M,\alpha)$ is a totally submanifold of Randers space $(\tilde{M}^{n+1},\tilde{\alpha}+\tilde{\beta})$. For more details on totally umbilical Finsler submanifolds one can refer to \cite{L}.
\begin{rem}\label{RIM}
Let $i:S\longrightarrow M$ be an isometric immersion. If $S$ is totally umbilical then the normal vector field $\xi$ is equal to the mean curvature vector field $\eta$.
\end{rem}

\subsection{Codazzi equation for Finsler submanifolds }
Consider a vector field $\hat{X}\in \Gamma(TTM_{0})$. We have locally $\hat{X}=X^{i}\frac{\delta}{\delta x^{i}}+\dot{X}^{i}\frac{\partial}{\partial y^{i}}$ where, $\{\frac{\delta}{\delta x^i},\frac{\partial}{\partial y^i}\}$ are horizontal and vertical bases of $TM$. Then one can define
\begin{align*}
Q:&\Gamma(TTM_{0})\longrightarrow \Gamma(TTM_{0}).\\
&Q\hat{X}:=\dot{X}^{i}\frac{\delta}{\delta x^{i}}+X^{i}\frac{\partial}{\partial y^{i}}.
\end{align*}
By means of the Cartan connection $\nabla$ on $(M,F)$ and the operator $Q$, one can define a linear connection on the manifold $TM_{0}$ by
\begin{equation*}
D_{\hat{X}}\hat{Y}:=\nabla_{\hat{X}}Y+Q\nabla_{\hat{X}}Q(H\hat{Y}),
\end{equation*}
for all $\hat{X},\hat{Y}\in \Gamma(TTM_{0})$. $D$ is called the \emph{associated linear connection} to $\nabla$ on $TM_{0}$. The torsion tensor field $T^{D}$ of $D$ is given by
\begin{equation*}
T^{D}(\hat{X},\hat{Y}):=\tau(\hat{X},\hat{Y})+Q(\nabla_{\hat{X}}Q(H\hat{Y})-\nabla_{\hat{Y}}Q(H\hat{X})-H[\hat{X},\hat{Y}]),
\end{equation*}
where, $\tau$ is the torsion tensor field of $\nabla$, cf., \cite{BF}. Let $R$ be the $hh$-curvature tensor of the Cartan connection $\nabla$, $\bar{\nabla}$ the induced connection on the submanifold $S$, $D$ the associated linear connection to the induced connection $\bar{\nabla}$ and $\bar{\nabla}^{^\perp}$ the linear connection on the normal bundle $N$. Let $A$ be the shape operator. One can define a covariant derivative $\nabla'$ of $A$ as follows, cf., \cite{BF}.
\begin{equation}\label{111}
(\nabla'_{\hat{X}}A)(W,\hat{Y}):=\bar{\nabla}_{\hat{X}}(A_{W}\hat{Y})-A_{\bar{\nabla}^{^\perp}_{\hat{X}}W}\hat{Y}-A_{W}(D_{\hat{X}}\hat{Y}),
\end{equation}
for any $\hat{X},\hat{Y}\in \Gamma(TTS_{0})$ and $W\in\Gamma(N)$. The \emph{$A$-Codazzi equation} for the Finsler submanifold $S$ with respect to the connection $\nabla$ on Finsler manifold $(M,F)$ is written
\begin{align}\label{codazzi}
g(R(X,Y)W,Z)&=g((\nabla'_{H\hat{Y}}A)(W,H\hat{X})-(\nabla'_{H\hat{X}}A)(W,H\hat{Y}),Z)\nonumber\\&-g(A_{W}(T^{D}(H\hat{X},H\hat{Y})),Z),
\end{align}
where, $W\in\Gamma(N)$, $X=\varrho(\hat{X})$, $Y=\varrho(\hat{Y})$ and $X,Y,Z\in \Gamma(q^{*}TS)$, cf., \cite{BF}, page 84.
\subsection{Sectional and flag curvatures}
Let $G_{2}(M)$ be the fiber bundle of 2-planes on $M$. Denote by $\pi^{-1}G_{2}(M)\longrightarrow SM$ the fiber induced on $SM$ by $\pi:SM\longrightarrow M$, where $SM$ is the unit sphere bundle. Let $P\in \pi^{-1}G_{2}(M)$ be a 2-plane generated by vectors $X,Y\in T_{x}M$ linearly independent at $x=\pi z\in M$ where, $z\in SM$. By means of $hh$-curvature tensors of Berwald and Cartan connection Akbar-Zadeh defined two \emph{sectional curvatures} denoted by $K_1$ and $K_2$ respectively. Here in this work we are dealing with Cartan connection and related sectional curvature $K_{2}:\pi^{-1}G_{2}(M)\longrightarrow \mathbb{R}$ defined by
\begin{equation}\label{OOO}
K_{2}(z,X,Y)=\frac{g(R(X,Y)Y,X)}{\parallel X\parallel^{2}\parallel Y\parallel^{2}-g(X,Y)^{2}},
\end{equation}
where, $R$ is the $hh$-curvature tensor of Cartan connection. The scalar $K_{2}$ is called the \emph{sectional curvature} at $z\in SM$. If the vector field $Y$ is replaced by the canonical section $v$ then sectional curvature is called flag curvature and does not depend on the choice of connection. If we denote the flag curvature by $K$ then we have
\begin{equation*}
K_{2}(z,v,X)=K(z,v,X),
\end{equation*}
where, $v$ is the canonical section, cf., \cite{AZ}, page 156.\\
Akbar-Zadeh as a \emph{generalization of Schur's theorem} has proved the following theorem.
\begin{theorem}\label{SS}
\cite{AZ} $K_{2}(z,P)$ is independent of 2-plane $P(X,Y)$ (dim $M>2$) if and only if the curvature tensor $R$ of the Cartan connection satisfies
\begin{equation*}
R(X,Y)Z=K[g(Y,Z)X-g(X,Z)Y],
\end{equation*}
where $K$ is a constant and $X,Y,Z\in T_{x}M$.
\end{theorem}
\section{Main results}
\begin{lem}\label{NNN}
Let $(M,F)$ be a Finsler manifold of dimension $n\geq 3$ satisfying the axiom of $r$-spheres for some $r$, $2\leq r< n$, then
\begin{equation*}
g(R(X,Y)Z,X)=0,
\end{equation*}
where, $X,Y,Z\in T_{x}M$ are three orthonormal vectors.
\end{lem}
\begin{proof} Let $(M,F)$ be a Finsler manifold which satisfies the axiom of $r$-spheres. Consider the Cartan connection $\nabla$ on the pulled-back bundle $p^{*}TM$, the induced connection $\bar{\nabla}$ on $S$ and the normal connection $\bar{\nabla}^{^\perp}$ on normal bundle. Let $X,Y$ and $Z$ be the three orthonormal vectors at $x=pz,z\in TM_0$. Consider the $r$-dimensional subspace $E_{r}$ of $T_{x}M$ which is normal to $Z$ and contains $X$ and $Y$. By assumption there exists an $r$-dimensional umbilical submanifold $S$ with parallel mean curvature vector field $\eta$ such that $x\in S$ and $T_{x}S=E_{r}$. It is well known for every point $x$ in a Finsler manifold there is a sufficiently small neighborhood $U$ on $M$ such that every pair of points in $U$ can be joined by a unique minimizing geodesic, see for instance \cite{BCS}, page 160. Hence there is a specific neighborhood $U$ of $x$ such that for each point $u\in U$ there exists a unique minimizing geodesic from $x$ to $u$. Let $W_{u}\in N_{u}S$ be the normal vector at $u$ which is parallel to $Z$ with respect to the normal connection $\bar{\nabla}^{^\perp}$ along the geodesic from $x$ to $u$ in $U$. The Finslerian metric $g$ on $TM_0$ defined by $F$, induces a Finslerian metric on $TS_0$, where we denote it again by $g$. By means of metric compatibility of Cartan connection, along each geodesic $\gamma$ from $x$ to any point in $U$ we have
\begin{equation}\label{ad}
\frac{d}{dt}g(W,\eta)=g(\nabla_{^h\!\hat{\dot{\gamma}}}W,\eta)+g(W,\nabla_{^h\!\hat{\dot{\gamma}}}\eta),
\end{equation}
where, $^h\!\hat{\dot{\gamma}}$ is the horizontal lift of the tangent vector field $\dot{\gamma}$. By means of the Weingarten formula (\ref{4}), rewrite (\ref{ad}) as follows
\begin{align}\label{add}
\frac{d}{dt}g(W,\eta)&=g(-A_{W}(^h\!\hat{\dot{\gamma}})+\bar{\nabla}^{^\perp}_{^h\!\hat{\dot{\gamma}}}W,\eta)+g(W,-A_{\eta}(^h\!\hat{\dot{\gamma}})+\bar{\nabla}^{^\perp}_{^h\!\hat{\dot{\gamma}}}\eta)\\
&=g(-A_{W}(^h\!\hat{\dot{\gamma}}),\eta)+g(\bar{\nabla}^{^\perp}_{^h\!\hat{\dot{\gamma}}}W,\eta)+g(W,-A_{\eta}(^h\!\hat{\dot{\gamma}}))+g(W,\bar{\nabla}^{^\perp}_{^h\!\hat{\dot{\gamma}}}\eta).\nonumber
\end{align}
Since $-A_{W}(^h\!\hat{\dot{\gamma}})$ and $-A_{\eta}(^h\!\hat{\dot{\gamma}})$ belong to $T_{x}S$ and on the other hand $\eta$ and $W$ are normal to $T_{x}S$ we have
\begin{equation*}
g(-A_{W}(^h\!\hat{\dot{\gamma}}),\eta)=g(W,-A_{\eta}(^h\!\hat{\dot{\gamma}}))=0.
\end{equation*}
By assumption the submanifold $S$ has parallel mean curvature vector field, that is, $\bar{\nabla}^{^\perp}_{^h\!\hat{\dot{\gamma}}}\eta=0$, hence $g(W,\bar{\nabla}^{^\perp}_{^h\!\hat{\dot{\gamma}}}\eta)=0$. By definition the vector $W$ is parallel along the geodesic $\gamma$ with respect to the normal connection $\bar{\nabla}^{^\perp}$, i.e. $\bar{\nabla}^{^\perp}_{^h\!\hat{\dot{\gamma}}}W=0$, hence  $g(\bar{\nabla}^{^\perp}_{^h\!\hat{\dot{\gamma}}}W,\eta)=0$. Therefore by means of (\ref{add}) we have $\frac{d}{dt}g(W,\eta)=0$ and $g(W,\eta)=\lambda$ is constant along each geodesic. Keeping in mind $S$ is a totally umbilical submanifold of $M$, we have $A_{W}=g(W,\eta)I=\lambda I$ at every point of $U$. Rewriting (\ref{111}) for the horizontal lift $^h\!\hat{X}$ of $X$ leads
\begin{equation}\label{222}
(\nabla'_{^h\!\hat{X}}A)(W,\hat{Y})=(\nabla^{*}_{^h\!\hat{X}}A_{W})(\hat{Y})-A_{\bar{\nabla}^{^\perp}_{^h\!\hat{X}}W}\hat{Y},
\end{equation}
where, we have put, $(\nabla^{*}_{^h\!\hat{X}}A_{W})(\hat{Y}):=\bar{\nabla}_{^h\!\hat{X}}(A_{W}\hat{Y})-A_{W}(D_{^h\!\hat{X}}\hat{Y})$ which can be considered as a covariant derivative of $A_{W}$. Plugging $A_{W}=\lambda I$ in the last equation leads
\begin{equation}\label{333}
\nabla^{*}_{^h\!\hat{X}}A_{W}=0.
\end{equation}
Similarly for the horizontal lift $^h\!\hat{Y}$ of $Y$ we have
\begin{equation}\label{444}
\nabla^{*}_{^h\!\hat{Y}}A_{W}=0.
\end{equation}
On the other hand, by means of metric compatibility of Cartan connection and the fact that $g(W,\eta)$ is constant we have $g(\nabla_{^h\!\hat{X}}W,\eta)+g(W,\nabla_{^h\!\hat{X}}\eta)=0$. By means of the Weingarten formula (\ref{4}) the last equation leads
\begin{equation}\label{LL}
g(-A_{W}(^h\!\hat{X}),\eta)+g(\bar{\nabla}^{^\perp}_{^h\!\hat{X}}W,\eta)+g(W,-A_{\eta}(^h\!\hat{X}))+g(W,\bar{\nabla}^{^\perp}_{^h\!\hat{X}}\eta)=0.
\end{equation}
Since $A_{W}(^h\!\hat{X})$ and $A_{\eta}(^h\!\hat{X})$ belong to $T_{x}S$ and on the other hand $\eta$ and $W$ are normal to $T_{x}S$ we have
\begin{equation*}
g(-A_{W}(^h\!\hat{X}),\eta)=g(W,-A_{\eta}(^h\!\hat{X}))=0.
\end{equation*}
By assumption the submanifold $S$ has parallel mean curvature vector field, that is, $\bar{\nabla}^{^\perp}_{^h\!\hat{X}}\eta=0$, hence $g(W,\bar{\nabla}^{^\perp}_{^h\!\hat{X}}\eta)=0$. Therefore (\ref{LL}) reduces to $g(\bar{\nabla}^{^\perp}_{^h\!\hat{X}}W,\eta)=0$. By  non-degeneracy of the metric tensor $g$ at $x\in S$ we have
\begin{equation}\label{5555}
\bar{\nabla}^{^\perp}_{^h\!\hat{X}}W=0.
\end{equation}
Similarly at $x\in S$ for vector $Y$ we obtain
\begin{equation}\label{555}
\bar{\nabla}^{^\perp}_{^h\!\hat{Y}}W=0.
\end{equation}
Therefore plugging (\ref{333}), (\ref{444}), (\ref{5555}) and (\ref{555}) in (\ref{222}) at $x\in S$ we obtain
\begin{equation*}
\nabla'_{^h\!\hat{X}}A=\nabla'_{^h\!\hat{Y}}A=0.
\end{equation*}
Now the Codazzi equation (\ref{codazzi}) implies
\begin{align}\label{codazzi2}
g(R(X,Y)W,X)=-g(A_{W}(T^{D}(^h\!\hat{X},^h\!\hat{Y})),X).
\end{align}
By assumption $A_{W}=g(W,\eta)I$. Thus we have
\begin{equation}\label{ne}
g(A_{W}(T^{D}(^h\!\hat{X},^h\!\hat{Y})),X)=g(T^{D}(^h\!\hat{X},^h\!\hat{Y}),X)g(W,\eta).
\end{equation}
Plugging (\ref{ne}) in (\ref{codazzi2}) we obtain
\begin{equation}\label{codazzi3}
g(R(X,Y)W,X)+g(T^{D}(^h\!\hat{X},^h\!\hat{Y}),X)g(W,\eta)=0.
\end{equation}
The first term $g(R(X,Y)W,X)$ is symmetric with respect to $Y$ and $W$, cf., \cite{AZ}, page 187. By means of the fact that $\eta$ is normal to $S$ we have $g(Y,\eta)=0$. Therefore we conclude
\begin{equation*}
g(T^{D}(^h\!\hat{X},^h\!\hat{Y}),X)g(W,\eta)=g(T^{D}(^h\!\hat{X},^h\!\hat{W}),X)g(Y,\eta)=0.
\end{equation*}
Thus (\ref{codazzi3}) becomes
\begin{equation*}
g(R(X,Y)W,X)=0.
\end{equation*}
Hence for orthonormal vectors $X,Y\in T_{x}S$ and $Z\in N_{x}S$ we have
\begin{equation*}
g(R(X,Y)Z,X)=0.
\end{equation*}
This completes the proof.
\end{proof}
\begin{lem}\label{MMMM}
Let $(M,F)$ be a Finsler manifold of dimension $n\geq 3$. If $g(R(X,Y)Z,X)=0$ whenever $X,Y$ and $Z$ are three orthonormal tangent vectors of $M$, then $M$ has constant flag curvature.
\end{lem}
\begin{proof}
If we put
\begin{equation*}
Y'=\frac{(Y+Z)}{\sqrt{2}}\quad,\quad Z'=\frac{(Y-Z)}{\sqrt{2}},
\end{equation*}
then since $X,Y$ and $Z$ are orthonormal, the vectors $X,Y'$ and $Z'$ are again orthonormal. By means of assumption
\begin{equation*}
g(R(X,Y')Z',X)=0.
\end{equation*}
By replacing $Y'$ and $Z'$ we obtain
\begin{equation}\label{1000}
g(R(X,Y)Y,X)=g(R(X,Z)Z,X).
\end{equation}
From which we can conclude from (\ref{OOO}), $K_{2}(z,X,Y)=K_{2}(z,X,Z)$. Thus the sectional curvature $K_{2}$ does not depend on the 2-plane $P(X,Y)$. By generalization of Schur's Theorem \ref{SS}, $M$ has constant sectional curvature and hence constant flag curvature. This completes the proof.
\end{proof}
{\bf Proof of Theorem \ref{main1}. } Let $(M,F)$ be a Finsler manifold which satisfies the axiom of $r$-spheres. By means of Lemmas \ref{NNN} and \ref{MMMM} we conclude that $M$ has constant flag curvature.
\hspace{\stretch{1}}$\Box$


{\small \it Corresponding author, Behroz Bidabad,  bidabad@aut.ac.ir}\\
{\small \it Maral Khadem Sedaghat,
 m\_sedaghat@aut.ac.ir \\ Faculty of Mathematics and Computer Science, Amirkabir University of Technology (Tehran Polytechnic), Hafez Ave., 15914 Tehran, Iran.}
\end{document}